\documentclass{article}  
\usepackage{graphicx} 
\usepackage{amssymb, latexsym, amsmath, amsthm} 
             
\usepackage{textcomp}
\usepackage[numbers]{natbib}
\usepackage[dvips=true,bookmarks=true]{hyperref} 
\usepackage{breakurl}
\usepackage{booktabs}
\usepackage{longtable}

\DeclareGraphicsExtensions{.pdf,.eps}
\usepackage{psfrag} 
\newtheorem{thm}{Theorem}[section]

\newtheorem{lem}[thm]{Lemma}

\theoremstyle{definition}
\newtheorem{def.}[thm]{Definition}

\newtheorem*{rem}{Remark} 

\def \Z {\mathbb{Z}}
\def \D {\mathbb{D}}
\def \N {\mathbb{N}}
\begin{document}

\title{Quasi-alternating Montesinos links}
\author{Tamara Widmer\\ Institut f\"ur Mathematik\\ Universit\"at Z\"urich}
\date{}
\maketitle
\begin{abstract}
The aim of this article is to detect new classes of quasi-alternating links. Quasi-alternating links are a natural generalization of alternating links. Their knot Floer and Khovanov homology are particularly easy to compute. Since knot Floer homology detects the genus of a knot  as well as whether a knot is fibered, as provided bounds on unknotting number and slice genus, characterization of quasi-alternating links becomes an interesting open problem. We show that there exist classes of non-alternating Montesinos links, which are quasi-alternating.
\end{abstract}

\section{Introduction}

Quasi-alternating links were introduced by  \citet*{Ozsvath:2003}. It was shown in  \cite{Manolescu:2007} that their knot Floer homology can be computed explicitly and depends only on the signature and the Alexander polynomial of the knot. More precisely it was shown that quasi-alternating links are homologically thin for both Khovanov homology and knot Floer homology. The definition is given in a recursive way:
\begin{def.}[ \cite{Ozsvath:2003}] \label{defqalt}
The set $\mathcal{Q}$ of \emph{quasi-alternating links} is the smallest set of links which satisfies the following properties:
\begin{itemize}
\item {The unknot is in $\mathcal{Q}$.}
\item {If the link $L$ has a diagram with a crossing $c$ such that} 
\begin{itemize}
\item[(i)] both smoothings of $c$, $L_0$ and $L_{\infty}$ as in Figure~\ref{F:skein}, are in $\mathcal{Q}$,
\item[(ii)] $\det(L_0), \det(L_\infty)\ne 0$,
\item[(iii)] $\det(L)=\det(L_0)+\det(L_{\infty})$;
\end{itemize}
then $L$ is in $\mathcal{Q}$.
The crossing $c$ is called a \emph{quasi-alternating crossing} of $L$ and $L$ is called \emph{quasi-alternating at $c$}.
\end{itemize}
\end{def.}

\begin{figure}[htbp]
	 \begin{center}
 	\psfrag{L0}[]{$L_0$}
 	\psfrag{L}[]{$L$}
 	\psfrag{Linf}[]{$L_\infty$}
 	\includegraphics[height=1.5cm]{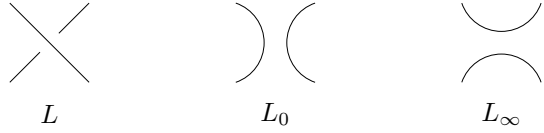}
	\end{center}
	 \caption{The link L at crossing $c$ and its resolutions $L_0$ and $L_{\infty}$.}
	 \label{F:skein}
\end{figure}

The class of quasi-alternating links contains all alternating links \cite{Ozsvath:2003}. It was shown by \citet*{Champanerkar:2007} that the sum of two quasi-alternating links is quasi-alternating and that a quasi-alternating crossing can be replaced by an alternating rational tangle to obtain another quasi-alternating link. Moreover they applied this result to show that there exist a family of pretzel links which is quasi-alternating.
\begin{figure}[htbp]
 \begin{center}
 \psfrag{R1}[]{$R_1$}
 \psfrag{R2}[]{$R_2$}
 \psfrag{Rm}[]{$R_m$}
 \psfrag{k}[]{$k$}
 \psfrag{an-1}[]{{\scriptsize $a_{n-1}$}}
 \psfrag{an}[]{{\scriptsize $a_n$}}
 \includegraphics[height=2.4cm]{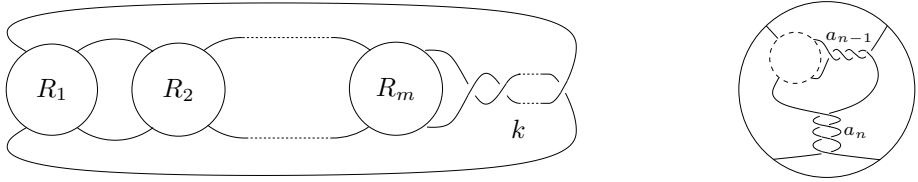}
\end{center}
 \caption{Montesinos link.}
 \label{F:montesinos}
 \end{figure}
We will apply their method to Montesinos links. A Montesinos link admits a diagram $D$ composed of $m\in\N$ rational tangle diagrams $R_1,R_2,\ldots,R_m$ and $k\in\N_0$ half-twists put together as in Figure~\ref{F:montesinos}. We will denote such a link by $L(R_1,R_2,\ldots,R_m;k)$. The rational tangles can be obtained from a sequence of non-zero integers $a_1,a_2,\ldots,a_n$ as indicated in Figure~\ref{F:montesinos}, and they are denoted by $R=a_1a_2\cdots a_n$. 
Our goal will be to prove the following theorem.
\begin{thm}\label{T:main}
Let $R=b_1b_2\cdots b_m$ represent a rational tangle with at least two crossings and let $a_i, b_i,c_i,m,n \in \N$ for all $i$ and $n\geq2$.
Then the following three Montesinos links yield infinite families of non-alternating, quasi-alternating links:
\begin{enumerate}
\item[(i)] $L\bigl (a_1a_2,R,-n \bigr )$ with $1+a_1(a_2-n)<0$,
\item[(ii)] $L\bigl (a_1a_2,R,(-c_1)(-c_2)\bigr )$ with $a_2<c_2$ or $a_2=c_2$ and $a_1>c_1$,
\item[(iii)] $L\bigl (a_1a_2a_3,R,-n \bigr )$ with $a_3<n$.
\end{enumerate}
\end{thm}

\section{Determinant}

The determinant of an alternating link is related to the number of spanning trees of its checkerboard graph. We will apply a generalization of this result obtained by \citet*{Dasbach:2006} to compute the determinant of Montesinos links. This for we recall the following definitions.
\begin{def.}[\cite{Lowrance:2007}]
The \emph{all-A dessin $\D(A)$} of a link, also known as \emph{ribbon graph}, is a graph which can be constructed out of a link diagram in the following way:\\
First, each crossing is replaced by an A-splicing (see Figure~\ref{F:splicing}). This results in a collection of circles in the plane with line segments joining them. Out of this projection, the all-A dessin is obtained by contracting each circle to a point such that the vertices of $\D(A)$ are in one-to-one correspondence with the circles. The edges then correspond to the line segments between them. The construction of the \emph{all-B dessin $\D(B)$} can be done analogously by replacing each crossing by a B-splicing.
\end{def.}
\begin{figure}[htbp]
 \begin{center}
 \psfrag{c}[]{crossing}
 \psfrag{Aspl}[]{A-splicing}
 \psfrag{Bspl}[]{B-splicing}
 \includegraphics[height=1.5cm]{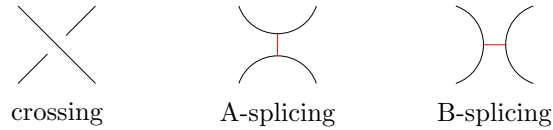}
\end{center}
 \caption{Splicings of a crossing.}
 \label{F:splicing}
 \end{figure}

If $\D$ is a dessin of a link $L$ there exists an orientation on it such that $\D$ can be viewed as a multi-graph equipped with a cyclic order on the edges at every vertex (for the exact construction see \cite{Lowrance:2007}). Therefore the dessin corresponds to a graph embedded on an orientable surface such that every region in the complement of the graph is a disc. We call the regions the \emph{faces} of the dessin. 

\begin{def.}[\cite{Lowrance:2007}]\label{def:dessingenus}
Let $\D$ be a dessin with one connected component and denote by $v(\D), e(\D)$ and $f(\D)$ the number of vertices, edges and faces in $\D$. 
The \emph{dessin genus} is calculated as follows:
$$g\left(\D(A)\right)=\frac{2-\left(v(\D)-e(\D)+f(\D)\right)}{2}.$$
\end{def.}
To compute $f(\D(A))$ the fact that $\D(A)$ and $\D(B)$ are dual to each other \citep{Lowrance:2007, Dasbach:2006a} is used. Hence $f(\D(A))=v(\D(B))$. Given these definitions a generalized formula to calculate the determinant of links with an all-A dessin of genus one can be stated.

\begin{thm}[\cite{Dasbach:2006}]  \label{thm:Dessingenus1}
Let $\D(A)$ and $\D(B)$ be the all-$A$ respectively the all-$B$ dessins of a connected link projection of a link $L$. Suppose $\D(A)$ is of dessin genus one.
Then
$$\det (L) =| \# \{\mbox{spanning trees in } \D(A)\} - \# \{\mbox {spanning trees in } \D(B)\}|.$$
\end{thm}

\begin{lem} \label{prop:redMontesino}
The dessin genus of a non-alternating Montesinos diagram equals one.
\end{lem}
\begin{proof} 
A non-alternating Montesinos link diagram can be obtained out of a non-alternating pretzel link diagram $P(p_1,\ldots,p_n,-q_1,\ldots,-q_m)$ by replacing the tassels with rational tangles. Inserting a rational tangle does not change the dessin genus, therefore the Montesinos link will have the same dessin genus as the pretzel link. It was shown in \cite{Champanerkar:2007} that this pretzel link diagrams have dessin genus one,thus the non-alternating Montesinos diagrams have dessin genus one too.
\end{proof}
According to this lemma the above theorem can be applied to Montesinos links. Moreover it shows that the Turaev genus of a non-alternating Montesinos link equals one.

\begin{figure}[htbp]
 \begin{center}
 \psfrag{a1}[]{$a_1$}
 \psfrag{a2}[]{$a_2$}
 \psfrag{a3}[]{$a_3$}
 \psfrag{a4}[]{$a_4$}
 \psfrag{423}[]{$C(4,,2,3)$}
 \psfrag{3124}[]{$C(3,-1,2,4)$}
 \includegraphics[height=2.5cm]{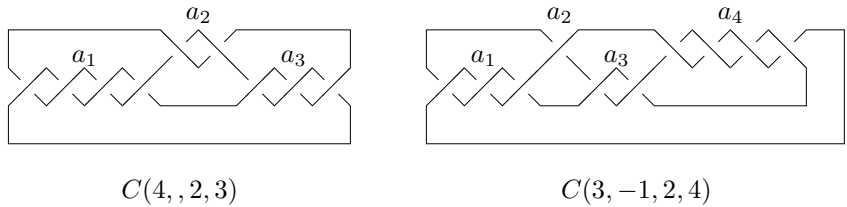}
\end{center}
 \caption[Rational links.]{Examples of rational links.}
 \label{F:ratlink}
 \end{figure}
A rational link is a link which admits a projection as in Figure~\ref{F:ratlink} and it is denoted by $C(a_1,a_2,\ldots,a_n)$. The determinant of these links has been studied by \citet*{Kauffman:2007}.

\begin {lem}[\cite{Kauffman:2007}] \label{lem:detratlink}
Let $n_i \in \Z \backslash \{0\}$ then we have
\begin{align}\notag
&\det C(n_1)&=&|n_1|&\\  \notag
&\det C(n_1,n_2)&=&|1+n_1n_2|&\\ \notag
&\det C(n_1,n_2,n_3)&=&|n_1+n_3+n_1n_2n_3|&\\ \notag
&\det C(n_1,n_2,n_3,n_4)&=&|1+n_1n_2+n_1n_4+n_3n_4+n_1n_2n_3n_4|.&
\end{align}
\end{lem}

\section{Quasi-alternating Montesinos links}

To obtain a family of non-alternating Montesinos links we will need the following definition.

\begin{def.}[\cite{Lickorish:1998}]\label{defreducedmontesino}
A diagram $D$ as in Figure~\ref{F:montesinos} is called a \emph{reduced Montesinos diagram} if it satisfies one of the following two conditions:
\begin{itemize}
\item[(i)] $D$ is alternating, or
\item[(ii)] Each $R_i$ is an alternating rational tangle diagram with at least two crossings placed in $D$ such that the two lower ends of $R_i$ belong to arcs incident to a common crossing  and $k=0$.
\end{itemize}
\end{def.}

It was shown by \citet*{Lickorish:1998} that a link which admits a $n$-crossing, reduced Montesinos diagram cannot be projected with fewer then $n$ crossings.

\begin{lem}
A link which admits a non-alternating reduced Montesinos diagram is non-alternating.
\end{lem}
\begin{proof}
Let $L$ be a link which admits a non-alternating, reduced Montesinos diagram with $n$ crossings. Therefore the minimal crossing number has to be $n$. Assume $L$ is alternating. Then $L$ possesses a connected, reduced, alternating diagram with $m$ crossings. According to a lemma of \citet{Lickorish:1997} $m$ is strictly smaller then any crossing number of a non-alternating diagram of the same link. Therefore $m<n$ which is a contradiction.
\end{proof}


\begin{proof}[\textbf{Proof of Theorem~\ref{T:main}}]
Let $R=b_1b_2\cdots b_m$ represent a rational tangle with at least two crossings and let $a_i, b_i,c_i,m,n \in \N$ for all $i$ and $n\geq2$.
Further let $L^1=L\bigl (a_1a_2,R,-n \bigr )$, $L^2=L\bigl (a_1a_2,R,(-c_1)(-c_2)\bigr )$ and $L^3=L\bigl (a_1a_2a_3,R,-n \bigr )$.
The links $L^1, L^2$ and $L^3$ are non-alternating since they possess a non-alternating, reduced Montesinos link diagram.\\
Now let $\widehat{L^i}$ be the link $L^i$ with the rational tangle $R$ replaced by one single positive crossing called $c$, for $i\in\{1,2,3\}$. First we show that the resolutions $\widehat{L^i_0}$ and $\widehat{L^i_\infty}$ are quasi-alternating at $c$. For all $i$, the link $\widehat{L^i_0}$ is the sum of two alternating links and therefore quasi-alternating.\\
The resolutions $\widehat{L^i_\infty}$ are rational links:
\begin{eqnarray*}
\widehat{L^1_\infty}&=&C\bigl(a_1,(a_2-n)\bigr)\\
\widehat{L^2_\infty}&=&C\bigl(a_1,(a_2-c_2),-c_1)\bigr)\\
\widehat{L^3_\infty}&=&C\bigl(a_1,a_2,(a_3-n)\bigr).
\end{eqnarray*}
Since each rational link possesses an alternating projection (see \citet*{Bankwitz:1934}), the resolutions $\widehat{L^i_\infty}$ are quasi-alternating. 
It remains to show that the determinants add up correctly.

\begin{itemize}
\item[(i)] For $\widehat{L^1}$ let $1+a_1(a_2-n)<0$. By applying Lemma~\ref{lem:detratlink}, the determinants of the resolutions for $\widehat{L^1}$ at $c$ hold:
\begin{eqnarray*}
\det(\widehat{L^1_0})&=&\det \bigl(T(2,-n)\#C(a_1,a_2)\bigr)=\det T(2,-n)\cdot \det C(a_1,a_2)\\&=&n(1+a_1a_2)\\
\det(\widehat{L^1_\infty})&=&\det C\bigl(a_1,(a_2-n)\bigr)=|1+a_1(a_2-n)|\\&=&(-1)(1+a_1(a_2-n)).
\end{eqnarray*}
The determinant for $\widehat{L^1}$ can be calculated out of the diagrams of its all-A and all-B dessins. The number of spanning trees can be computed directly by inspecting the diagrams of Figure~\ref{F:ABL1}:

\begin{figure}[htbp]
\begin{center}
\psfrag{Bsplice}[]{all-B splicing}
\psfrag{Bdessin}[]{all-B dessin}
\psfrag{Asplice}[]{all-A splicing}
\psfrag{Adessin}[]{all-A dessin}
\psfrag{a1}[]{{\scriptsize$a_1$}}
\psfrag{a2}[]{{\scriptsize $a_2$}}
\psfrag{n}[]{{\scriptsize $n$}}
\includegraphics[height=7cm]{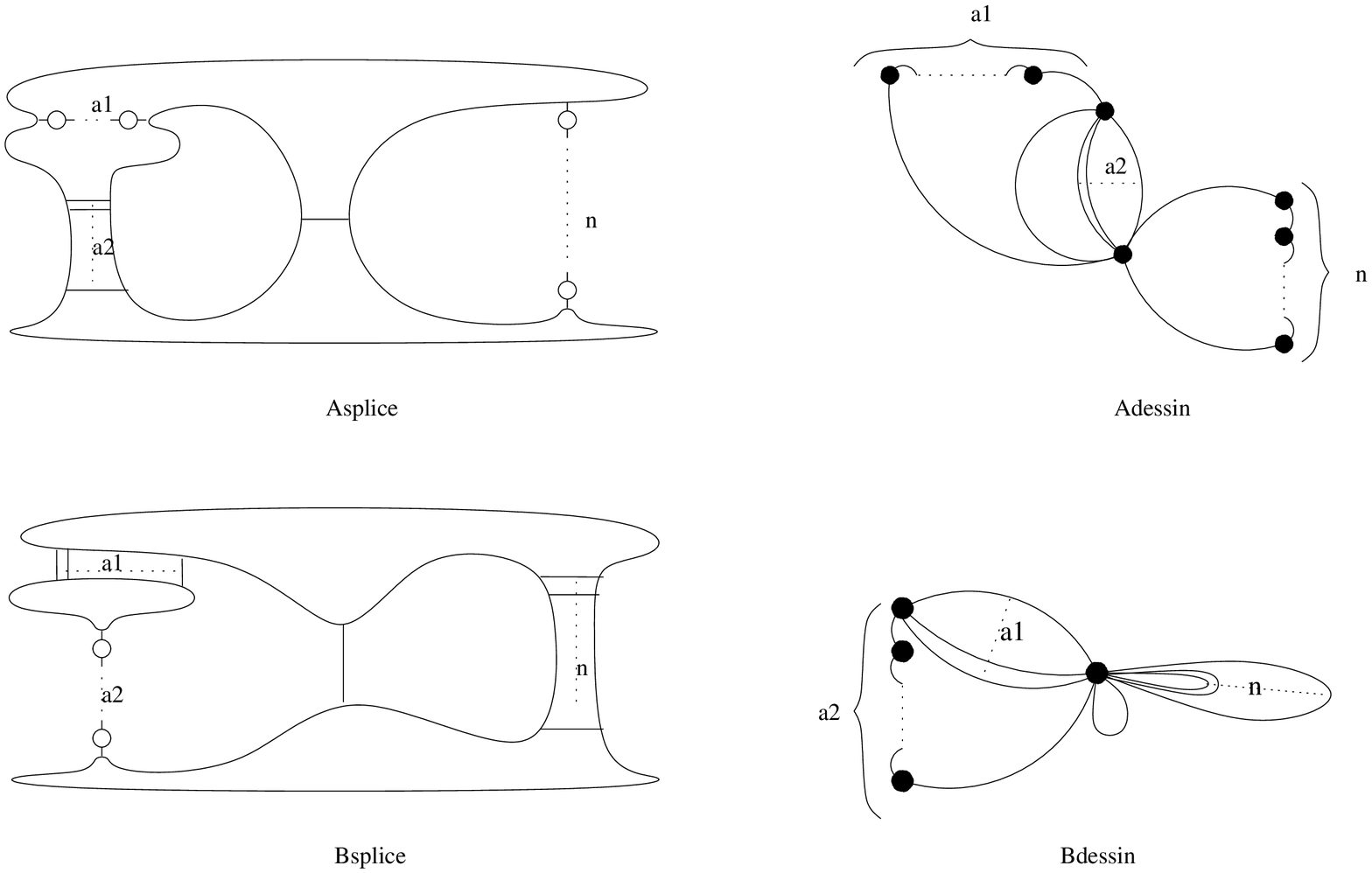}
\end{center}
\caption{All-A/B splicings and the all-A/B dessins of $\widehat{L^1}=L\bigl(a_1a_2,1,-n\bigr)$.}
\label{F:ABL1}
\end{figure}

\begin{eqnarray*}
\# \{\mbox{spanning trees in } \D(A)\} &=&n\bigl(a_1(a_2+1)+1\bigr)\\
\# \{\mbox{spanning trees in } \D(B)\}&=&a_1a_2+1.
\end{eqnarray*}
By Theorem~\ref{thm:Dessingenus1} we get $\det(L)=|n(a_1a_2+a_1+1)-a_1a_2-1|$. Hence
\begin{eqnarray*}
\det(L_0)+\det(L_\infty)&=&\underbrace{n(1+a_1a_2)}_{>0}+\underbrace{(-1)(1+a_1(a_2-n))}_{>0}\\
&=&|n(a_1a_2+a_1+1)-a_1a_2-1|\\
&=&\det(L).
\end{eqnarray*}
\item[(ii)] For $\widehat{L^2}$ let $a_2<c_2$ or $a_2=c_2$ and $a_1>c_1$. For the determinants of the resolutions we get:
\begin{eqnarray*}
\det(\widehat{L^2_0})&=&\det \bigl(C(a_1,a_2)\#C(c_1,c_2)\bigr)=\det C(a_1,a_2)\cdot \det C(c_1,c_2)\\&=&(1+a_1a_2)(1+c_1c_2)\\
\det(\widehat{L^2_\infty})&=& \bigl |a_1-c_1-a_1c_1\underbrace{(a_2-c_2)}_{\le 0} \bigr |\\ &=&a_1-c_1-a_1c_1(a_2-c_2)\\ &=&a_1(1+c_1c_2)-c_1(1+a_1a_2).
\end{eqnarray*}

\begin{figure}[htbp]
\begin{center}
\psfrag{Bsplice}[]{all-B splicing}
\psfrag{Bdessin}[]{all-B dessin}
\psfrag{Asplice}[]{all-A splicing}
\psfrag{Adessin}[]{all-A dessin}
\psfrag{a1}[]{{\scriptsize $a_1$}}
\psfrag{a2}[]{{\scriptsize $a_2$}}
\psfrag{c1}[]{{\scriptsize $c_1$}}
\psfrag{c2}[]{{\scriptsize $c_2$}}
\includegraphics[height=7cm]{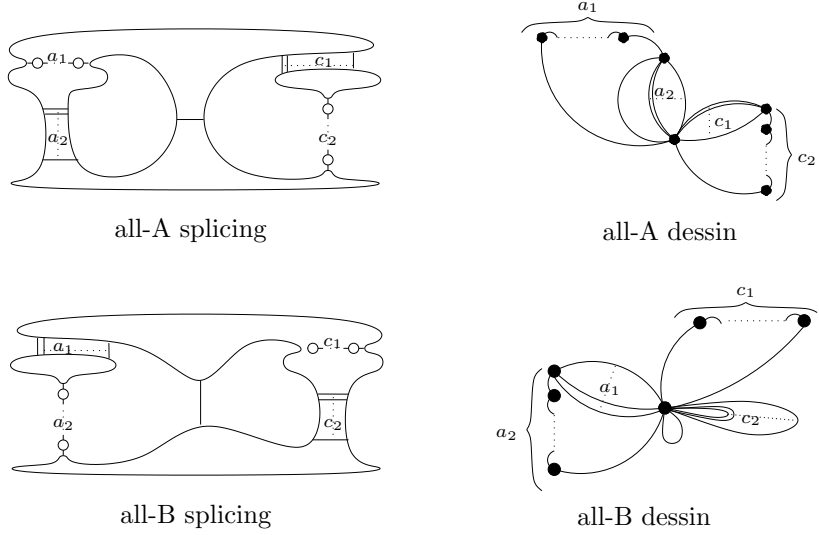}
\end{center}
\caption{All-A/B splicings and the all-A/B dessins of $\widehat{L^2}=L\bigl(a_1a_2,1,(-c_1)(-c_2)\bigr)$.}
\label{F:ABL2}
\end{figure}

The number of spanning trees of the all-A/B dessins of $\widehat{L^2}$ (see Figure~\ref{F:ABL2}) is given by:
\begin{eqnarray*}
\# \{\mbox{spanning trees in } \D(A)\} &=&(a_1a_2+a_1+1)(c_1c_2+1)\\
\# \{\mbox{spanning trees in } \D(B)\}&=&c_1(a_1a_2+1).
\end{eqnarray*}
Now according to Theorem~\ref {thm:Dessingenus1} we get:
\begin{eqnarray*}
\det(L)&=&\bigr |(a_1a_2+a_1+1)(c_1c_2+1)-c_1(a_1a_2+1)\bigr |\\
&=& \bigl |(a_1a_2+1)(c_1c_2+1)+a_1(c_1c_2+1)-c_1(a_1a_2+1)\bigr |\\
&=& \bigl |(a_1a_2+1)(c_1c_2+1)\bigl |+ \bigr |a_1(c_1c_2+1)-c_1(a_1a_2+1)\bigr |\\
&=&\det(L_0)+\det(L_\infty).
\end{eqnarray*}
\item[(iii)] For $\widehat{L^3}$ let $a_3<n$. For the determinants of the resolutions we get:
\begin{eqnarray*}
\det(\widehat{L^3_0})&=&\det \bigl(C(a_1,a_2,a_3)\#T(2,-n)\bigr)=n(a_1+a_3+a_1a_2a_3)\\
\det(\widehat{L^3_\infty})&=&\det C(a_1,a_2,a_3-n)=\bigl |a_1+(a_3-n)+a_1a_2(a_3-n)\bigr |\\
&=&\bigl |a_1+\underbrace{(a_3-n)}_{<0}(1+a_1a_2)\bigr |\\
&=&(-1)\bigl (a_1+a_3+a_1a_2a_3-n(1+a_1a_2)\bigr ).
\end{eqnarray*}

\begin{figure}[htbp]
\begin{center}
\psfrag{a1}[]{{\scriptsize$a_1$}}
\psfrag{a2}[]{{\scriptsize $a_2$}}
\psfrag{a3}[]{{\scriptsize $a_3$}}
\psfrag{n}[]{{\scriptsize $n$}}
\psfrag{Bsplice}[]{all-B splicing}
\psfrag{Bdessin}[]{all-B dessin}
\psfrag{Asplice}[]{all-A splicing}
\psfrag{Adessin}[]{all-A dessin}
\includegraphics[height=7cm]{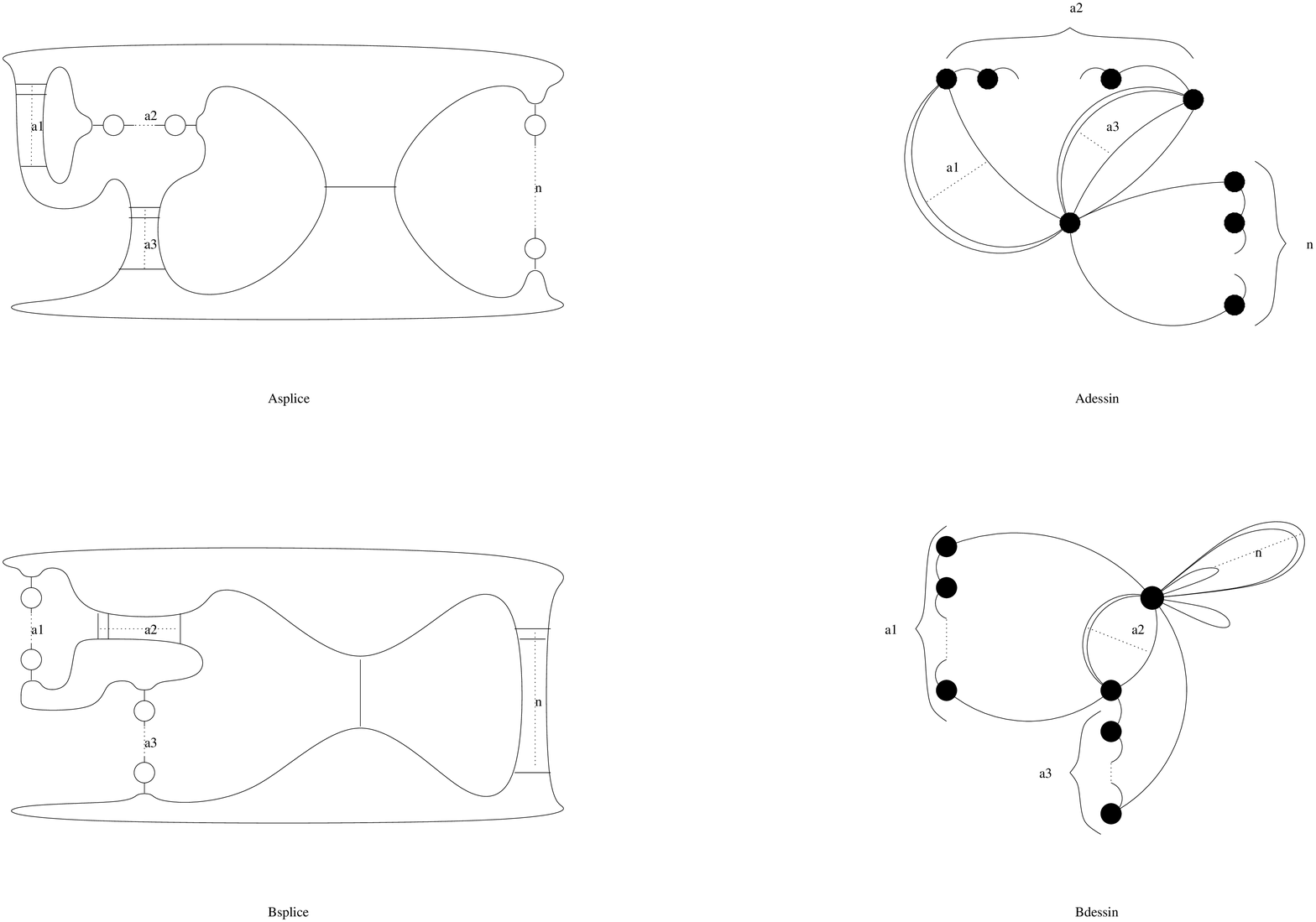}
\end{center}
\caption{All-A/B splicings and the all-A/B dessins of $\widehat{L^3}=L(a_1a_2a_3,1,-n)$.}
\label{F:ABL3}
\end{figure}

The number of spanning trees of the all -A/B dessin of $\widehat{L^3}$ (see Figure~\ref{F:ABL3}) is given by:
\begin{eqnarray*}
\# \{\mbox{spanning trees in } \D(A)\} &=&n(1+a_1+a_3+a_1a_2+a_1a_2a_3)\\
\# \{\mbox{spanning trees in } \D(B)\}&=&a_1+a_3+a_1a_2a_3,
\end{eqnarray*}
which leads to
\begin{eqnarray*}
\det(L)&=&\bigl |n(1+a_1+a_3+a_1a_2+a_1a_2a_3)-(a_1+a_3+a_1a_2a_3)\bigr |\\
&=&\bigl |\underbrace{n(a_1+a_3+a_1a_2a_3)}_{>0}+\underbrace{(-1)(a_1+a_3+a_1a_2a_3-n(1+a_1a_2))}_{>0}\bigr |\\
&=&\det(L_0)+\det(L_\infty).
\end{eqnarray*}
\end{itemize}
This shows that the determinants add up correctly for all three links $\widehat{L^i}$. Since all the resolutions are quasi-alternating the crossing at $c$ is quasi-alternating. Therefore, according to \cite{Champanerkar:2007}, it can be replaced by any alternating rational tangle which extends $c$. This completes the proof.
\end{proof}

\begin{rem}
For $a_i, b_i, c_i\in\N$ and $R=b_1b_2\cdots b_m$ it can be shown that the link $L\bigl(a_1a_2a_3,R,(-c_1)(-c_2)(-c_3)\bigr)$ with $\frac{1+a_1a_2}{1+c_1c_2}>\frac{a_1+a_3+a_1a_2a_3}{c_1+c_3+c_1c_2c_3}$ is quasi-alternating. The proof is analogous to the previous one only it has to be taken in account that the sum of two rational links is always a rational link whose determinant can be calculated out of the continuous fraction which is defined by the link \cite{Ernst:1990}.  It is notable that the calculation of the spanning trees gets more complicated the more tassels the link has.
\end{rem}

\section{Examples}

We will now apply Theorem ~\ref{T:main} to analyze knots with 11 crossings.
There exist 185 non-alternating prime knots with crossing number 11. Out of these, 67 are Montesinos links.  By our method, we can identifiy 23 non-alternating Montesinos links to be quasi-alternating. These are listed in Table~1, together with their Conway notation according to KnotInfo~\cite{Cha:}. Note that the minus at the last rational tangle represents a negative crossing. We have the following identities: $[2-]=[2,-1]=[-2]$, $[21-]=[21,-1]=[-3]$ and $[3-]=[3,-1]=[(-2)(-1)]$.
\begin{table}[htbp]
\begin{center}
  \setlength{\belowcaptionskip}{6pt}
\caption{Non-alternating, quasi-alternating Montesinos knots with 11 crossings detected to be quasi-alternating by Theorem~\ref{T:main}.}
\begin{tabular} {l l|| l l }\toprule
Knot		&Conway notation	&Knot	&Conway notation\\ \midrule
11n2		&[221;211;2-]		&11n84	&[22;22;21-]\\
11n3		&[221;22;2-]		&11n87	&[212;21;21-]\\
11n14	&[41;211;2-]		&11n89	&[31;211;21-]\\
11n15	&[41;22;2-]		&11n90	&[31;211;3-]\\
11n17	&[311;211;2-]		&11n100	&[221;3;21-]\\
11n18	&[311;22;2-]		&11n103	&[211;211;21-]\\
11n29	&[231;21;2-]		&11n106	&[212;3;21-]\\
11n30	&[231;3;2-]		&11n122	&[32;3;21-]\\
11n48	&[31;22;3-]		&11n137	&[311;21;21-]\\
11n63	&[411;21;2-]		&11n140	&[41;21;21-]\\
11n64	&[411;3;2-]		&11n141	&[41;3;3-]\\
11n83	&[31;22;21-]		
\end{tabular}
\end{center}
  \setlength{\belowcaptionskip}{0pt}
\end{table}

By applying the findings of  \citet*{Champanerkar:2007}, we can identify 17 more Montesinos links to be quasi-alternating. In Table~2 for each knot there is one rational tangle of the Conway notation indicated in bold. This tangle is replaced with a crossing of the same sign and checked if it is a quasi-alternating crossing in the new diagram.
\begin{table}[htbp]
\begin{center}
  \setlength{\belowcaptionskip}{6pt}
\caption{Non-alternating, quasi-alternating Montesinos knots with 11 crossings detected to be quasi-alternating by inserting a rational tanlge.}
\begin{tabular} {l l || l l}\toprule
Knot			&Conway notation  	&Knot	&Conway notation \\ \midrule
11n1		&[\textbf{23};211;2-]		&11n58	&[\textbf{312};21;2-]\\	
11n13	&[\textbf{5};211;2-]		&11n59	&[\textbf{3111};21;2-]\\
11n16	&[\textbf{32};211;2-]		&11n60	&[3111;\textbf{3};2-]\\	
11n28	&[\textbf{24};21;2-]		&11n62	&[\textbf{42};21;2-]\\
11n51	&[\textbf{213};21;2-]		&11n82	&[\textbf{4};22;21-]\\	
11n52	&[\textbf{2121};21;2-]	&11n91	&[\textbf{4};211;21-]\\
11n54	&[\textbf{2112};21;2-]	&11n101	&[\textbf{23};21;21-]\\	
11n55	&[\textbf{21111};21;2-]	&11n105	&[\textbf{2111};21;21-]\\
11n56	&[21111;\textbf{3};2-]
\end{tabular}
\end{center}
  \setlength{\belowcaptionskip}{0pt}
\end{table}

For the sake of completeness, Table~3 gives a  list of all non-alternating, quasi-alternating knots up to 10 crossings detected by \citet*{Manolescu:2006a}, \citet*{Baldwin:2008} and \citet*{Champanerkar:2007}. The 16 knots indicated with a cross could also be detected to be quasi-alternating by Theorem~\ref{T:main}. The remaining non-alternating knots are Khovanov homologically thick except for $9_{46}$ and $10_{140}$, which are not quasi-alternating by forthcoming work of Shumakovitch.

\begin{table}[htbp] 
\begin{center}
  \setlength{\belowcaptionskip}{6pt}
\caption{Non-alternating, quasi-alternating knots up to 10 crossings.}
\begin{tabular} {l r l |l r l |l r l |l r l}\toprule
$8_{20}$	 	& \cite{Baldwin:2008}		&&$8_{21}$	& \cite {Manolescu:2006a}&x&
$9_{43}$		& \cite {Manolescu:2006a}	&x&$9_{44}$	& \cite {Manolescu:2006a}&\\
$9_{45}$		& \cite {Manolescu:2006a}	&x&$9_{47}$	& \cite {Manolescu:2006a}&&
$9_{48}$		& \cite {Manolescu:2006a}	&x&$9_{49}$	& \cite {Manolescu:2006a}&\\
$10_{125}$	& \cite {Baldwin:2008}		&&$10_{126}$	& \cite {Baldwin:2008}&x&
$10_{127}$	& \cite {Baldwin:2008}		&x&$10_{129}$	& \cite {Champanerkar:2007}&\\
$10_{130}$	& \cite {Champanerkar:2007}	&x&$10_{131}$	& \cite {Champanerkar:2007}&x&
$10_{133}$	& \cite {Champanerkar:2007}	&&$10_{134}$		& \cite {Champanerkar:2007}&x\\
$10_{135}$	& \cite {Champanerkar:2007}	&x&$10_{137}$	& \cite {Champanerkar:2007}&&
$10_{138}$	& \cite {Champanerkar:2007}	&x&$10_{141}$	& \cite {Baldwin:2008}&\\
$10_{142}$	& \cite {Champanerkar:2007}	&x&$10_{143}$	& \cite {Baldwin:2008}&x&
$10_{144}$	& \cite {Champanerkar:2007}	&x&$10_{146}$	& \cite {Champanerkar:2007}&x\\
$10_{147}$	& \cite {Champanerkar:2007}	&x&$10_{148}$	& \cite {Baldwin:2008}&&
$10_{149}$	& \cite  {Baldwin:2008}		&&$10_{150}$	& \cite {Champanerkar:2007}&\\
$10_{151}$	& \cite {Champanerkar:2007}	&&$10_{155}$	& \cite  {Baldwin:2008}&&
$10_{156}$	& \cite {Champanerkar:2007}	&&$10_{157}$	& \cite  {Baldwin:2008}&\\
$10_{158}$	& \cite {Champanerkar:2007}	&&$10_{159}$	& \cite {Baldwin:2008}&&
$10_{160}$	& \cite {Champanerkar:2007}	&&$10_{163}$	& \cite {Champanerkar:2007}&\\
$10_{164}$	& \cite {Champanerkar:2007}	&&$10_{165}$	& \cite {Champanerkar:2007}&&
$10_{166}$	& \cite {Champanerkar:2007}	&&			&	&\\ \bottomrule
\end{tabular}
\end{center}
  \setlength{\belowcaptionskip}{0pt}
\end{table}

\subsection{Acknowledgements}
I am deeply grateful to Anna Beliakova, whose help was essential for me to develop the presented results.


\end{document}